\documentclass{spie_fickus}

\usepackage{amsthm}
\usepackage{amssymb}
\usepackage{amsmath}
\usepackage{graphicx}

\usepackage[colorlinks=true]{hyperref}					

\newcommand{\bbR}{\mathbb{R}}

\newcommand{\rmd}{\mathrm{d}}

\newcommand{\rmI}{\mathrm{I}}
\newcommand{\rmN}{\mathrm{N}}

\newcommand{\Tr}{\mathrm{Tr}}
\newcommand{\MSE}{\mathrm{MSE}}

\newcommand{\abs}[1]{|{#1}|}

\newcommand{\set}[1]{\{{#1}\}}

\newcommand{\norm}[1]{\|{#1}\|}

\newcommand{\ip}[2]{\langle{#1},{#2}\rangle}

\newtheorem{theorem}{Theorem}

\theoremstyle{definition}
\newtheorem{problem}[theorem]{Problem}

\newtheorem{definition}[theorem]{Definition}

\title{Frame completions for optimally robust reconstruction} 

\author{Matthew Fickus\supit{a}, Dustin G.~Mixon\supit{b} and Miriam J.~Poteet\supit{a}
\skiplinehalf
\supit{a}Department of Mathematics and Statistics, Air Force Institute of Technology\\Wright-Patterson Air Force Base, Ohio 45433, USA\\
\supit{b}Program in Applied and Computational Mathematics, Princeton University\\Princeton, New Jersey 08544, USA
}
\authorinfo{Send correspondence to Matthew Fickus: E-mail: Matthew.Fickus@afit.edu}
 
\begin{document} 
\maketitle

\begin{abstract}
In information fusion, one is often confronted with the following problem: given a preexisting set of measurements about an unknown quantity, what new measurements should one collect in order to accomplish a given fusion task with optimal accuracy and efficiency.  We illustrate just how difficult this problem can become by considering one of its more simple forms: when the unknown quantity is a vector in a Hilbert space, the task itself is vector reconstruction, and the measurements are linear functionals, that is, inner products of the unknown vector with given measurement vectors.  Such reconstruction problems are the subject of frame theory.  Here, we can measure the quality of a given frame by the average reconstruction error induced by noisy measurements; the mean square error is known to be the trace of the inverse of the frame operator.  We discuss preliminary results which help indicate how to add new vectors to a given frame in order to reduce this mean square error as much as possible.
\end{abstract}



\section{Introduction}

The \textit{synthesis operator} of a sequence of vectors $F=\set{f_n}_{n=1}^N$ in $\bbR^M$ is $F:\bbR^N\rightarrow\bbR^M$, $\smash{Fy:=\sum_{n=1}^N y(n)f_n}$.  That is, $F$ is the $M\times N$ matrix whose columns are the $f_n$'s.  Here and throughout, we make no notational distinction between the vectors themselves and the synthesis operator they induce.  The vectors $F$ are said to be a \textit{frame} for $\bbR^M$ if there exists \textit{frame bounds} $0<A\leq B<\infty$ such that $A\norm{x}^2\leq\norm{F^*x}^2\leq B\norm{x}^2$ for all $x\in\bbR^M$.  In this finite-dimensional setting, the optimal frame bounds $A$ and $B$ of an arbitrary $F$ are the least and greatest eigenvalues of the \textit{frame operator}:
\begin{equation*}
FF^*=\sum_{n=1}^{N}f_n^{}f_n^*,
\end{equation*}
respectively.  Here, $f_n^*$ is the linear functional $f_n^*:\bbR^M\rightarrow\bbR$, $f_n^*f:=\ip{f}{f_n}$.  In particular, we have that $F$ is a frame if and only if the $f_n$'s span $\bbR^M$, which necessitates $M\leq N$.

Frames provide numerically stable methods for finding overcomplete decompositions of vectors, and as such are useful tools in various signal processing applications~\cite{KovacevicC:07a,KovacevicC:07b}.  Indeed, if $F$ is a frame, then any $x\in\bbR^M$ can be decomposed as:
\begin{equation*}
x=F\tilde{F}^*x=\sum_{n=1}^{N}\ip{x}{\tilde{f}_n}f_n,
\end{equation*}
where $\tilde{F}=\set{\tilde{f}_n}_{n=1}^{N}$ is a \textit{dual frame} of $F$, meaning it satisfies $F\tilde{F}^*=\rmI$.  The most often-used dual frame is the \textit{canonical} dual, namely the pseudoinverse $\tilde{F}=(FF^*)^{-1}F$.  Note that computing a canonical dual involves the inversion of the frame operator.  As such, when designing a frame for a given application, it is important to retain control over the spectrum $\set{\lambda_m}_{m=1}^{M}$ of $FF^*$.  Here and throughout, such spectra are arranged in nonincreasing order, with the optimal frame bounds $A$ and $B$ being $\lambda_M$ and $\lambda_1$, respectively.

The \textit{mean square error} is one way to measure the quality of a given dual frame $\tilde{F}$.  In particular, consider the problem of reconstructing $x$ from $F^*x+\epsilon$, where $\epsilon$ is an additive error term.  Applying any given dual $\tilde{F}$ to $F^*x+\epsilon$ gives $\tilde{F}(F^*x+\epsilon)=\tilde{F}F^*x+\tilde{F}\epsilon=x+\tilde{F}\epsilon$.  The magnitude of the error between the original vector $x$ and its reconstructed approximation $x+\tilde{F}\epsilon$ is $\norm{\tilde{F}\epsilon}$.  Treating $\epsilon$ as a random variable, the corresponding \textit{mean square error} (MSE) is the expectation of $\norm{\tilde{F}\epsilon}^2$ with respect to $\epsilon$'s probability density function $p$:
\begin{equation*}
\MSE[\tilde{F}]
:=\int_{\bbR^N}\norm{\tilde{F}\epsilon}^2 p(\epsilon)\,\rmd\epsilon.
\end{equation*}
In the special case where $\tilde{F}$ is the canonical dual $(FF^*)^{-1}F$ and the entries of $\epsilon$ are independently distributed with each having mean zero and variance $\sigma^2$, the MSE can be simplified~\cite{GoyalVT:98} in terms of the trace of the inverse of the frame operator:
\begin{equation}
\label{equation.definition of MSE}
\MSE[(FF^*)^{-1}F]
=\sigma^2\Tr[(FF^*)^{-1}]
=\sigma^2\sum_{m=1}^M \frac{1}{\lambda_m}.
\end{equation}
In order to construct reconstruction systems that are robust to noise, we try to design frames for which the MSE \eqref{equation.definition of MSE} is as small as possible.  One (nonrealistic) way to do this is to simply scale one's frame vectors $F$.  Indeed for any $c>0$, the frame operator $c^2FF^*$ of $cF$ has eigenvalues $\set{c^2\lambda_m}_{m=1}^{M}$ and so the corresponding MSE tends to zero as $c$ becomes large.  However, this method for lessening MSE is unrealistic in real-world communications applications, since there the magnitudes $\set{\abs{\ip{x}{f_n}}}_{n=1}^{N}$ of the entries of the transmitted signal $F^*x$ are bounded in terms of the signal-to-noise ratio of the given channel.  As such, the problem of minimizing the MSE is usually considered in the context of frames $F$ in which the $n$th frame element $f_n$ is required to have a given prescribed length $\norm{f_n}^2=\mu_n$.

In particular, choosing $\mu_n=1$ for all $n=1,\dotsc,N$, one may consider the problem of minimizing the MSE \eqref{equation.definition of MSE} over all $M\times N$ matrices $F$ with unit norm columns.  This problem has been solved~\cite{GoyalVT:98}, the answer being that such an $F$ is necessarily a \textit{tight frame}, meaning its lower and upper frame bounds $A$ and $B$ can be taken to be equal, namely $FF^*=A\rmI$ for some $A>0$.  Such \textit{unit norm tight frames} (UNTFs) have the property that their $N\times N$ \textit{Gram matrices} $F^*F$ have ones along their diagonal, and so their $M\times M$ frame operators $FF^*=A\rmI$ satisfy:
\begin{equation*}
MA
=\Tr(A\rmI)
=\Tr(FF^*)
=\Tr(F^*F)
=N.
\end{equation*}
Thus, the tight frame constant $A$ of a UNTF is necessarily $A=\frac NM$.  As such, the MSE of a UNTF $F$ is:
\begin{equation}
\label{equation.MSE of UNTF}
\sigma^2\sum_{m=1}^M \frac{1}{\lambda_m}
=\sigma^2\sum_{m=1}^M \frac{M}{N}
=\frac{(M\sigma)^2}{N}.
\end{equation}
Note that due to the $\frac1N$ term in \eqref{equation.MSE of UNTF}, for any fixed $M$ and $\sigma$ this MSE will tend to zero as $N$ grows large.  That is, for a given signal $x$ and a fixed signal-to-noise-per-transmitted-coefficient parameter $\sigma$, one may expect to communicate $x$ with arbitrarily high levels of reliability even through a noisy channel, provided one is willing to first encode $x$ via an $N\times M$ matrix $F^*$ of a correspondingly high level of redundancy.  That is, this reliability is purchased at the cost of computing and transmitting $F^*x$, which is a longer signal than $x$ itself.

We are interested in generalizing these ideas to the realm of information fusion.  There, we assume that we have already been given $F_{N_0}^*x$ where $F_{N_0}=\set{f_n}_{n=1}^{N_0}$ is a fixed known $M\times N_0$ matrix.  That is, we assume that we are given a set of inner products $\set{\ip{x}{f_n}}_{n=1}^{N_0}$.  If $F_{N_0}$ is a frame, then $x$ can be reconstructed from these measurements.  However, this reconstruction may lack stability if $F_{N_0}$ is poorly conditioned; in such cases the  MSE of the canonical dual of $F_{N_0}$ is large.  We therefore seek to add measurement vectors to this frame---to \textit{complete} $F_{N_0}=\set{f_n}_{n=1}^{N_0}$ to a longer matrix $F_{N}=\set{f_n}_{n=1}^{N}$ where $N>N_0$---in a manner so that the MSE of the canonical dual of $F_N$ is minimal.  Here, we again formulate the problem realistically by prescribing the norms of these new measurements.  To be precise, we are interested in solving the following problem:

\begin{problem}
\label{problem.main problem}
Given vectors $F_{N_0}=\set{f_n}_{n=1}^{N_0}$ in $\bbR^M$, an integer $N>N_0$ and a sequence of desired norms $\set{\mu_n}_{n=N_0+1}^{N}$, find vectors $\set{f_n}_{n=N_0+1}^{N}$ in $\bbR^M$ that have the property that the mean square error $\Tr[(F_N^{}F_N^*)^{-1}]$ of the resulting completed frame $F_N=\set{f_n}_{n=1}^{N}$ is minimal.
\end{problem}

In the remainder of this paper, we discuss some of our recent progress towards solving this problem.


\section{Constructing frames with eigensteps}

Solving Problem~\ref{problem.main problem} will require us to have a good understanding of how the spectrum of a frame operator can change as a result of the inclusion of new frame elements of known prescribed norms.  We have recently made progress~\cite{CahillFMPS:11,FickusMPS:11} in solving a less-difficult version of this problem: Given nonnegative nonincreasing sequences $\set{\lambda_m}_{m=1}^{M}$ and $\set{\mu_n}_{n=1}^{N}$, construct all frames $F=\set{f_n}_{n=1}^{N}$ with the property that $FF^*$ has spectrum $\set{\lambda_m}_{m=1}^{M}$ and that $\norm{f_n}^2=\mu_n$ for all $n$.  This work is based on the classical notion of eigenvalue interlacing.

To be precise, we say that a given nonnegative sequence $\set{\beta_m}_{m=1}^{n}$ \textit{interlaces} on another such sequence $\set{\alpha_m}_{m=1}^{n-1}$, denoted $\set{\alpha_m}_{m=1}^{n-1}\sqsubseteq\set{\beta_m}_{m=1}^{n}$, provided $\beta_{m+1}\leq\alpha_m\leq\beta_m$ for all $m=1,\dotsc,n-1$.  This notion of interlacing can be extended to sequences of identical length: we say that $\set{\alpha_m}_{m=1}^{M}\sqsubseteq\set{\beta_m}_{m=1}^{M}$ provided $\alpha_M\leq\beta_M$ and $\beta_{m+1}\leq\alpha_m\leq\beta_m$ for all $m=1,\dotsc,M-1$.  It is classically known that if $G$ is self-adjoint, then the spectrum of $G+ff^*$ interlaces on that of $G$; given a sequence of vectors $F=\set{f_n}_{n=1}^{N}$, we apply this fact to the frame operators of the partial sequences $F_n=\set{f_m}_{m=1}^{n}$:
\begin{equation*}
F_n^{}F_n^*=\sum_{m=1}^{n}f_m^{}f_m^*.
\end{equation*}
Letting $\set{\lambda_{n;m}}_{m=1}^{M}$ denote the spectrum of $F_n^{}F_n^{*}$, we have that $F_{n+1}^{}F_{n+1}^{*}=F_{n}^{}F_{n}^{*}+f_{n+1}^{}f_{n+1}^{*}$ and so $\set{\lambda_{n+1;m}}_{m=1}^{M}$ interlaces on $\set{\lambda_{n;m}}_{m=1}^{M}$.  Such a sequence of interlacing spectra is known~\cite{CahillFMPS:11,FickusMPS:11} as a sequence of \textit{(outer) eigensteps}:
\begin{definition}
\label{definition.outer eigensteps}
Let $\set{\lambda_m}_{m=1}^{M}$ and $\set{\mu_n}_{n=1}^{N}$ be nonnegative and nonincreasing.  A corresponding sequence of \textit{outer eigensteps} is a sequence of sequences $\set{\set{\lambda_{n;m}}_{m=1}^{M}}_{n=0}^{N}$ which satisfies the following four properties:
\begin{enumerate}
\renewcommand{\labelenumi}{(\roman{enumi})}
\item $\lambda_{0;m}=0$ for every $m=1,\ldots,M$,
\item $\lambda_{N;m}=\lambda_m$ for every $m=1,\ldots,M$,
\item $\set{\lambda_{n-1;m}}_{m=1}^{M}\sqsubseteq\set{\lambda_{n;m}}_{m=1}^{M}$ for every $n=1,\ldots,N$,
\item $\sum_{m=1}^{M}\lambda_{n;m}=\sum_{m=1}^{n}\mu_{m}$ for every $n=1,\ldots,N$.
\end{enumerate}
\end{definition}
The following result, namely Theorem~$2$ of Cahill \textit{et al}'s work~\cite{CahillFMPS:11}, characterizes the existence of a sequence of vectors whose frame operator possesses a given desired spectrum and whose elements have given desired lengths:
\begin{theorem}
\label{theorem.necessity and sufficiency of eigensteps}
For any nonnegative nonincreasing sequences $\set{\lambda_m}_{m=1}^{M}$ and $\set{\mu_n}_{n=1}^{N}$, every sequence of vectors $F=\set{f_n}_{n=1}^{N}$ in $\bbR^M$ whose frame operator $FF^*$ has spectrum $\set{\lambda_m}_{m=1}^{M}$ and which satisfies $\norm{f_n}^2=\mu_n$ for all $n$ can be constructed by the following process:
\begin{enumerate}
\renewcommand{\labelenumi}{\Alph{enumi}.}
\item
Pick outer eigensteps $\set{\set{\lambda_{n;m}}_{m=1}^{M}}_{n=0}^{N}$ as in Definition~\ref{definition.outer eigensteps}. 
\item
For each $n=1,\dotsc,N$, consider the polynomial:
\begin{equation*}
p_n(x):=\prod_{m=1}^{M}(x-\lambda_{n;m}).
\end{equation*}
Take any $f_1\in\bbR^M$ such that $\norm{f_1}^2=\mu_1$.

For each $n=1,\dotsc,N-1$, choose any $f_{n+1}$ such that:
\begin{equation}
\label{equation.necessity and sufficiency of eigensteps 2}
\norm{P_{n;\lambda}f_{n+1}}^2=-\lim_{x\rightarrow\lambda}(x-\lambda)\frac{p_{n+1}(x)}{p_n(x)}\qquad \forall \lambda\in\set{\lambda_{n;m}}_{m=1}^M.
\end{equation}
Here, $P_{n;\lambda}$ denotes the orthogonal projection operator onto the eigenspace $\rmN(\lambda\rmI-F_n^{}F_n^*)$ of the frame operator of $F_n:=\set{f_m}_{m=1}^{n}$.  The limit in~\eqref{equation.necessity and sufficiency of eigensteps 2} necessarily exists and is nonpositive.
\end{enumerate}
Conversely, any $F$ constructed by this process has $\set{\lambda_m}_{m=1}^{M}$ as the spectrum of $FF^*$ and $\norm{f_n}^2=\mu_n$ for all $n$, and moreover, $F_n^{}F_n^*$ has spectrum $\set{\lambda_{n;m}}_{m=1}^{M}$.
\end{theorem}
The method of Theorem~\ref{theorem.necessity and sufficiency of eigensteps} is the first known algorithm for producing all such frames.  However, there are two issues with this method with respect to implementation.  We discuss the first issue in Section $3$ and the second issue in Section $4$.


\section{Characterizing eigensteps}

Our first issue with the algorithm of Theorem~\ref{theorem.necessity and sufficiency of eigensteps} is that Step~A is vague.  Indeed, for a given $\set{\lambda_m}_{m=1}^{M}$ and $\set{\mu_n}_{n=1}^{N}$, it is not clear whether or not a given sequence of outer eigensteps even exists, let alone how one should find them all.  Fortunately, these issues can be addressed~\cite{FickusMPS:11}.  The key idea is to transition from the eigenvalues $\set{\lambda_{n;m}}_{m=1}^{M}$ of the $M\times M$ frame operator $F_n^{}F_n^{*}$ to the eigenvalues $\set{\lambda_{n;m}}_{m=1}^{n}$ of the $n\times n$ Gram matrix $F_n^*F_n^{}$.  These two spectra are zero-padded versions of each other.  Under this transition, the notion of a sequence of outer eigensteps (Definition \ref{definition.outer eigensteps}) transforms into an alternative but equivalent notion of \textit{inner eigensteps}:
\begin{definition}
\label{definition.inner eigensteps}
Let $\set{\lambda_n}_{n=1}^{N}$ and $\set{\mu_n}_{n=1}^{N}$ be nonnegative nonincreasing sequences.  A corresponding sequence of \textit{inner eigensteps} is a sequence of sequences $\set{\set{\lambda_{n;m}}_{m=1}^{n}}_{n=1}^{N}$ which satisfies the following three properties:
\begin{enumerate}
\renewcommand{\labelenumi}{(\roman{enumi})}
\item $\lambda_{N;m}=\lambda_m$ for every $m=1,\ldots,N$,
\item $\set{\lambda_{n-1;m}}_{m=1}^{n-1}\sqsubseteq\set{\lambda_{n;m}}_{m=1}^{n}$ for every $n=2,\dotsc,N$,
\item $\sum_{m=1}^{n}\lambda_{n;m}=\sum_{m=1}^{n}\mu_{m}$ for every $n=1,\dotsc,N$.
\end{enumerate}
\end{definition}
Note the terms ``outer" and ``inner" follow from the matrices that they are derived from: outer eigensteps arise as spectra of frame operators, which are sums of outer products; inner eigensteps arise from Gram matrices, which are tables of inner products.  In the next result, we formally verify that one may naturally identify a sequence of outer eigensteps with a sequence of inner eigensteps, and vice versa, provided one zero-pads appropriately:
\begin{theorem}
\label{theorem.inner vs outer}
Given nonnegative nonincreasing sequences $\{\lambda_n\}_{n=1}^N$ and $\{\mu_n\}_{n=1}^N$, where $\lambda_n=0$ for every $n>M$, every choice of inner eigensteps corresponds to a unique choice of outer eigensteps and vice versa, the two being zero-padded versions of each other.
Specifically, inner eigensteps $\{\{\lambda_{n;m}\}_{m=1}^n\}_{n=1}^N$ correspond to outer eigensteps $\{\{\lambda_{n;m}\}_{m=1}^M\}_{n=0}^N$, where $\lambda_{n;m}=0$ whenever $n=0$ or $m>n$.
Conversely, outer eigensteps $\{\{\lambda_{n;m}\}_{m=1}^M\}_{n=0}^N$ correspond to inner eigensteps $\{\{\lambda_{n;m}\}_{m=1}^n\}_{n=1}^N$, where $\lambda_{n;m}=0$ whenever $m>M$.

Moreover, for every $n=1,\ldots,N$, $\{\lambda_{n;m}\}_{m=1}^M$ is the spectrum of the frame operator $F_n^{}F_n^*$ of $F_n=\{f_m\}_{m=1}^n$ if and only if $\{\lambda_{n;m}\}_{m=1}^n$ is the spectrum of the Gram matrix $F_n^*F_n^{}$.
\end{theorem}

\begin{proof}
First, take outer eigensteps $\{\{\lambda_{n;m}\}_{m=1}^M\}_{n=0}^N$, and consider $\{\{\lambda_{n;m}\}_{m=1}^n\}_{n=1}^N$, where we define
\begin{equation}
\label{eq.out2in 1}
\lambda_{n;m}=0\qquad\mbox{whenever } m>M.
\end{equation}
Then Definition~\ref{definition.inner eigensteps}.i follows from Definition~\ref{definition.outer eigensteps}.ii when $m\leq M$, and from \eqref{eq.out2in 1} and the assumption that $\lambda_n=0$ for every $n>M$ when $m>M$.
Next, we note that Definition~\ref{definition.outer eigensteps}.iii gives
\begin{align}
\label{eq.out2in 2}
\lambda_{n;m+1}
&\leq\lambda_{n-1;m}\leq\lambda_{n;m}\qquad\forall m=1,\ldots,M-1,\\
\label{eq.out2in 3}
0
&\leq\lambda_{n-1,M}\leq\lambda_{n;M},
\end{align}
for every $n=1,\ldots,N$.
To prove Definition~\ref{definition.inner eigensteps}.ii, pick any $n=2,\ldots,N$.
We need to show $\lambda_{n;m+1}\leq\lambda_{n-1;m}\leq\lambda_{n;m}$ for every $m=1,\ldots,n-1$.
This follows directly from \eqref{eq.out2in 2} when $n\leq M$ or when $n>M$ and $m<M$.
If $n>M$ and $m=M$, then \eqref{eq.out2in 1} and \eqref{eq.out2in 3} together give
\begin{equation*}
\lambda_{n;M+1}=0\leq\lambda_{n-1,M}\leq\lambda_{n;M}.
\end{equation*}
Also, \eqref{eq.out2in 1} gives that $\lambda_{n;m+1}\leq\lambda_{n-1;m}\leq\lambda_{n;m}$ becomes $0\leq0\leq0$ whenever $n>M$ and $m>M$.
For Definition~\ref{definition.inner eigensteps}.iii, note that when $n\geq M$, \eqref{eq.out2in 1} and Definition~\ref{definition.outer eigensteps}.iv together give
\begin{equation}
\label{eq.out2in 5}
\sum_{m=1}^n\lambda_{n;m}=\sum_{m=1}^M\lambda_{n;m}=\sum_{m=1}^n\mu_m.
\end{equation}
Furthermore, if $n<M$, then Definition~\ref{definition.outer eigensteps}.i and Definition~\ref{definition.outer eigensteps}.iii together give
\begin{equation}
\label{eq.out2in 4}
\lambda_{n;m}=0\qquad\mbox{whenever } m>n,
\end{equation}
and so \eqref{eq.out2in 4} and Definition~\ref{definition.outer eigensteps}.iv together give \eqref{eq.out2in 5}.

Now take inner eigensteps $\{\{\lambda_{n;m}\}_{m=1}^n\}_{n=1}^N$, and consider $\{\{\lambda_{n;m}\}_{m=1}^M\}_{n=0}^N$, where we define
\begin{equation}
\label{eq.in2out 1}
\lambda_{n;m}=0\qquad\mbox{whenever } m>n.
\end{equation}
Then Definition~\ref{definition.outer eigensteps}.i follows directly from \eqref{eq.in2out 1} by taking $n=0$.
Also, Definition~\ref{definition.outer eigensteps}.ii follows from Definition~\ref{definition.inner eigensteps}.i since $M\leq N$.
Next, Definition~\ref{definition.inner eigensteps}.ii gives 
\begin{equation}
\label{eq.in2out 2}
\lambda_{n;m+1}\leq\lambda_{n-1;m}\leq\lambda_{n;m}\qquad\forall n=2,\ldots,N,~m=1,\ldots,n-1.
\end{equation}
Using the nonnegativity of $\{\lambda_n\}_{n=1}^N$ along with Definition~\ref{definition.inner eigensteps}.i and an iterative application of the left-hand inequality of \eqref{eq.in2out 2} then gives
\begin{equation*}
0\leq\lambda_N=\lambda_{N;N}\leq\cdots\leq\lambda_{n;n}\qquad\forall n=1,\ldots,N.
\end{equation*}
Combining this with an iterative application of the right-hand inequality of \eqref{eq.in2out 2} then gives
\begin{equation}
\label{eq.in2out 3}
0\leq\lambda_{m;m}\leq\cdots\leq\lambda_{n;m}\qquad\forall n\geq m.
\end{equation}
For Definition~\ref{definition.outer eigensteps}.iii, we need to show \eqref{eq.out2in 2} and \eqref{eq.out2in 3} for every $n=1,\ldots,N$.
Considering \eqref{eq.in2out 1}, when $n=1$, \eqref{eq.out2in 2} and \eqref{eq.out2in 3} together become $\lambda_{1;1}\geq0$, which follows from \eqref{eq.in2out 3}.
Also, when $n>M$, \eqref{eq.in2out 2} immediately gives \eqref{eq.out2in 2}, while \eqref{eq.out2in 3} follows from both \eqref{eq.in2out 3} and \eqref{eq.in2out 2}:
\begin{equation*}
0\leq\lambda_{n;M+1}\leq\lambda_{n-1;M}\leq\lambda_{n;M}.
\end{equation*}
For the case $2\leq n\leq M$, note that \eqref{eq.in2out 2} gives the inequalities in \eqref{eq.out2in 2} whenever $m\leq n-1$.
Furthermore when $m=n$, \eqref{eq.in2out 1} gives $\lambda_{n;n+1}=\lambda_{n-1;n}=0$, and so the inequalities in \eqref{eq.out2in 2} become $\lambda_{n;n}\geq0$, which follows from \eqref{eq.in2out 3}.
Otherwise when $m>n$, the inequalities in \eqref{eq.out2in 2} become $0\leq0\leq0$ by \eqref{eq.in2out 1}.
To finish the case $2\leq n\leq M$, we need to prove \eqref{eq.out2in 3}.
When $n=M$, \eqref{eq.out2in 3} becomes $\lambda_{n;n}\geq 0$, which follows from \eqref{eq.in2out 3}.
Otherwise when $n<M$, \eqref{eq.out2in 3} becomes $0\leq0\leq0$ by \eqref{eq.in2out 1}.
We now have only to prove Definition~\ref{definition.outer eigensteps}.iv.
For $n\leq M$, \eqref{eq.in2out 1} and Definition~\ref{definition.inner eigensteps}.iii together imply
\begin{equation}
\label{eq.in2out 5}
\sum_{m=1}^M\lambda_{n;m}=\sum_{m=1}^n\lambda_{n;m}=\sum_{m=1}^n\mu_m.
\end{equation}
Next, note that \eqref{eq.in2out 3}, Definition~\ref{definition.inner eigensteps}.i, and our assumption that $\lambda_n=0$ for every $n>M$ gives
\begin{equation*}
0\leq\lambda_{n;m}\leq\lambda_{N;m}=\lambda_m=0\qquad\mbox{whenever } n\geq m>M.
\end{equation*}
Thus, $\lambda_{n;m}=0$ whenever $n\geq m> M$; when $n>M$, we can combine this with Definition~\ref{definition.inner eigensteps}.iii to get \eqref{eq.in2out 5}.
\end{proof}
This result, when coupled with a complete constructive characterization of all valid outer eigensteps---itself a nontrivial result---provides a systematic method for constructing any and all valid inner eigensteps~\cite{FickusMPS:11}, thereby making Step~A of Theorem~\ref{theorem.necessity and sufficiency of eigensteps} explicit.


\section{Constructing frames from eigensteps}

Our second issue with the algorithm of Theorem~\ref{theorem.necessity and sufficiency of eigensteps} is that Step~B apparently requires a large amount of tedious linear algebra.  To be precise, in order to implement Step~B, one appears forced to compute the eigenvectors of $F_n^{}F_n^*$ for each $n=1,\dotsc,N-1$.  Though this is not difficult---by construction, the eigenvalues of $F_n^{}F_n^{*}$ are known to be $\set{\lambda_{n;m}}_{m=1}^{M}$---this does not lend itself to elegant closed-form expressions for the frame vectors themselves.  Fortunately, this process can be made surprisingly explicit: see Theorem~$7$ of Cahill \textit{et al}'s work~\cite{CahillFMPS:11}.  We do not present these details here.  However, for the sake of the interested reader, we do provide the following MATLAB code that implements this improved version of Step~B.  

Here, for the sake of simplicity, $V_n=\rmI$ for all $n$.  The following two functions must be placed in the same directory in order to execute the code.   The first function \verb|constructU.m| implements Step~B.  Here the output of the function is a slightly modified version of Steps~B.4 and B.5.  The function \verb|constructU.m| returns $U_n^*f_{n+1}^{}$ and $U_n^*U_{n+1}^{}$. We recursively call \verb|constructU.m| and then multiply the ouputs at each iteration in order to calculate $f_{n}$ and $U_{n}$ for $n=2,\dots,N$.  This step is accomplished by the function \verb|constructFrame.m| which outputs the final sequence of vectors $F=\set{f_n}_{n=1}^N$ whose frame operator $FF^*$ has spectrum $\set{\lambda_m}_{m=1}^M$ and which satisfies $\|f_n\|^2 = \mu_n$ for all $n$.

\begin{verbatim}
function [U, Uf] = constructU(E1, E2)
% Description:  This function implements Steps B.1-5 of the algorithm to 
%               explicitly construct any and all sequences of vectors whose
%               partial-frame operator spectra match the eigensteps chosen
%               in Step A. Here, we assume V1,...,Vn are the identity. 
% Call:     	[U, Uf] = constructU(E1, E2) 
%                   E1 = Spectrum at (n)
%                   E2 = Spectrum at (n+1)
% Output:       U = U_(n)* U_(n+1)
%               Uf = U_(n)* f_(n+1)
% File:         constructU.m

%spectra must be row vectors and listed in descending order*****************
if ~isrow(E1), E1=E1'; end
if ~isrow(E2), E2=E2'; end
E1 = sort(E1, 'descend');
E2 = sort(E2, 'descend');

M = length(E1);
        
%Find indices of unique elements (Step B.2)********************************
R1 = E1; %Unique set of eigenvalues of E1
R2 = E2; %Unique set of eigenvalues of E2
for i = 1:M
    [tf, loc] = ismember(E1(i), R2);
    if tf == 1
        [tf,loclast]=ismember(E1(i), R1);
        R1(loclast) = -1;
        R2(loc)=-1;
    end
end

%Index sets of unique elements of E1 and E2, respectively.
I = find(R1 >=0);
J = find(R2 >=0);

M1 = length(I);
M2 = M-M1;

R1 = R1(I);
R2 = R2(J);

%Construct column and row vectors (Step B.3)*******************************
for i = 1:M1
    P(i) = sqrt(-prod(R1(i)-R2)/prod(R1(i)-R1(find(R1~=R1(i)))));
    Q(i) = sqrt(prod(R2(i)-R1)/prod(R2(i)-R2(find(R2~=R2(i)))));
end

%Construct difference matrix
for i = 1:M1
    D(i,:) = 1./(R2-R1(i));
end
W =(P'*Q).*D;

%Create Block Diagonal Matrix
UU=blkdiag(W,eye(M2));

%Permute Row and Columns;
PRow = permMat(I,M);
PCol = permMat(J,M);

%Compute U and Uf (Modified Steps B.4 and B.5)*****************************
U = PRow*UU*inv(PCol);
Uf=zeros(M,1);
Uf(I) = P;
%**************************************************************************
%Define permutation matrix given the unique index set I 
function P = permMat(I,M)
m = 1:M;
pi = [I setdiff(m,I)];
Id = eye(M);
P = Id(1:M,pi);

%**************************************************************************
function tf = isrow(E)
[rows,cols]=size(E);
if rows==1
    tf=1;
else
    tf=0;
end
\end{verbatim}

\begin{verbatim}
function F = constructFrame(E,U1)
% Description:  This function implements Step B of the algorithm to 
%               explicitly construct any and all sequences of vectors whose
%               partial-frame operator spectra match the eigensteps chosen
%               in Step A. Here, we assume V1,...,Vn are the identity. 
% Call:     	F = constructFrame(E, U1) 
%                   E = Matrix of eigensteps 
%                   U1 = Initial unitary matrix
% Output:       The frame, F.
% File:         constructFrame.m

UU(:,:,1) = U1;
U(:,:,1) = UU(:,:,1);
F(:,1) = U(:,1,1);
[M,N] = size(E);

for i = 2:N
    [UU(:,:,i), Uf(:,i)] = constructU(E(:,i-1),E(:,i));
    U(:,:,i) = eye(M);

    %Multiply matrices to find new U
    for j = 1:i, U(:,:,i) = U(:,:,i)*UU(:,:,j);end
    
    %Multiply to find new f
    Temp = eye(M);
    for j = 1:i-1; Temp = Temp*UU(:,:,j); end
    F(:,i) = Temp*Uf(:,i);
end
\end{verbatim}

The following example reproduces the results of Example~$8$ of Cahill \textit{et al}'s work~\cite{CahillFMPS:11}.

\begin{verbatim}
>> E = [0 0 0 2/3 5/3;0 1/3 4/3 5/3 5/3;1 5/3 5/3 5/3 5/3]

E =

         0         0         0    0.6667    1.6667
         0    0.3333    1.3333    1.6667    1.6667
    1.0000    1.6667    1.6667    1.6667    1.6667

>> U1=eye(3)

U1 =

     1     0     0
     0     1     0
     0     0     1

>> F = constructFrame(E,U1)

F =

    1.0000    0.6667   -0.4082   -0.1667    0.1667
         0    0.7454    0.9129    0.3727   -0.3727
         0         0         0    0.9129    0.9129
\end{verbatim}

\section*{Acknowledgments}
This work was supported by NSF DMS 1042701, NSF CCF 1017278, AFOSR F1ATA01103J001, AFOSR F1ATA00183G003 and the A.~B.~Krongard Fellowship.  The views expressed in this article are those of the authors and do not reflect the official policy or position of the United States Air Force, Department of Defense, or the U.S.~Government.

\end{document}